\documentclass[12pt]{amsart}
\usepackage{amsmath,amsxtra,amssymb,latexsym, amscd,amsthm,amsfonts}
\newtheorem{thm}{Theorem}[section]
\newtheorem{cor}[thm]{Corollary}%%
\newtheorem{lem}[thm]{Lemma}
\newtheorem{prop}[thm]{Proposition}
\theoremstyle{definition}
\newtheorem{defn}[thm]{Definition}
\theoremstyle{remark}
\newtheorem{rem}[thm]{Remark}
\numberwithin{equation}{section}
\newcommand{\R}{\mathbb R}
\newcommand{\N}{\mathbb N}

\newcommand{\Om}{\mathcal O}

\newcommand{\simuleq}[1]{\left\{\begin{aligned}#1\end{aligned}\right.}%

\begin{document}%

\title[]{HOPF-LAX FORMULA  and generalized characteristics}%
\author{NGUYEN HOANG}\thanks{This research is partially supported by
the  NAFOSTED, Vietnam, under grant \# 101.02-2013.09 and by the project DHH2013-03-35.
}

\address{Department of Mathematics, College of Education, Hue University, 32 LeLoi, Hue, Viet Nam}%
\email{nguyenhoanghue@gmail.com     or: nguyenhoang@hueuni.edu.vn}%


\begin{abstract}
We study some differential properties of viscosity solution  for Hamilton - Jacobi equations defined by Hopf-Lax formula 
$u(t,x)=\min_{y\in \R^n} \big\{ \sigma
(y)+tH^*\big (\frac {x-y}{t}\big)\big \}.$ A generalized form of characteristics of the Cauchy problem is taken into account the context. Then we examine the strip of differentiability of the viscosity solution given by the function $u(t,x).$
\end{abstract}
\maketitle
% ----------------------------------------------------------------
\section{Introduction}

The global studies of Hamilton-Jacobi equations have spent more than a half century with several notions of solutions including classical and generalized solutions. Among explicit representations of the solutions, two formulas Hopf-Lax and Hopf play important roles. Many qualitative as well as quantitative properties of solutions are studied via these formulas. The first formula was established in 1957 by Lax \cite{la} for 1-dimensional case and then in 1965, Hopf \cite{h} proved both formulas for multidimensional case. Until now, they turned to be classical. Nevertheless, recently several papers pursuit to exploit the mentioned formulas and get new results, e.g., \cite{au,aa},...
\smallskip

Consider the Cauchy problem for Hamilton-Jacobi
equations of the form
 \begin{equation}\label{1.1}\frac{\partial u}{\partial t} + H(t,x,D_x u)=0\, , \,\, (t,x)\in \Omega=(0,T)\times \R^n,
\end{equation}
\begin{equation}\label{1.2}
u(0,x)=\sigma(x)\, , \,\, x\in \R^n.\end{equation}

If the Hamiltonian is independent of  $(t,x)$ i.e., $H(t,x,p)=H(p)$  and $H(p)$ is convex and superlinear, $\sigma$ is Lipschitz on
$\R^n,$ then the Hopf-Lax formula for solution of the problem (\ref{1.1})-(\ref{1.2})  is defined by
\begin{equation}\label{1.3}u(t,x)=\min_{y\in \R^n} \Big\{ \sigma
(y)+tH^*\big (\frac {x-y}{t}\big)\Big \},\end{equation}
here $H^*$ denotes
the Fenchel conjugate of the convex function $H.$ 

\smallskip
This formula gives the unique viscosity solution to the above Cauchy problem \cite{be,ev, cs, st}.
If $H=H(t,x,p)$ is a convex function with respect to the last variable, then the representation of  viscosity solution of the problem (\ref{1.1})-(\ref{1.2}) (which can be considered as a general form of Hopf-Lax formula)  is established and studied thoroughly; see \cite{ev,cs} and references therein. Especially, in the monograph \cite{cs}, the analysis of regularity properties of the viscosity solution is studied comprehensively where method of calculus of variation is used.
\smallskip

This paper is devoted to investigate some differential properties of formula (\ref{1.3}) for a class of convex Hamiltonians $H=H(p).$ We define and then use the ``generalized characteristics'' for nonsmooth initial data to connect Hopf-Lax formula, the way is somewhat similar to the case of classical solution of the problem since the solution $u(t,x)$ may be of class $C^1$ whereas the initial may not be. Although several differential results are known in the general cases for $H=H(t,x,p)$ under strong assumptions, e.g., $H(t,x,p)$ is of class $C^2$ and the Hessian $D^2 H$ is definite positive;  see \cite{cs}, it is not convenient to apply them to the case $H=H(p)$ under minimum hypotheses. Therefore, in this paper we attempt to exploit tools and technique of Calculus and PDEs to deal with the specific case and not to infer any result from Calculus of Variation. Several properties here can not directly deduced by the general ones.
\smallskip

The structure of the paper is as follows. In section 2, we first collect some facts related to semidifferentials of a function  and
viscosity solution of the Cauchy problem for Hamilton-Jacobi equations defined by Hopf-Lax formula for convex/concave Hamiltonians. We show that the Hopf-Lax formula inherits the semiconvexity of initial data in short time under suitable assumption on $H.$  In section 3, we introduce a kind of generalized characteristics of the Cauchy problem using the semidifferentials of the intial data. Then we present some differential properties of Hopf-Lax formula related to the generalized characteristics.
\smallskip

In section 4, we study the regularity of viscosity solution $u(t,x)$ defined by Hopf-Lax formula and verify some sufficient conditions such that the function $u(t,x)$  is differentiable on the regions of the form $(0,t_*)\times\R^n.$ Finally, we rededuce a result on the backward and forward solution based on Hopf-Lax formula, presented in \cite{bcjs} as an application of this section.

\medskip
 We use the following notations. Let  $T$ be a positive number, $\Omega =(0,T)\times
\R^n;\, |\, .\, |$ and $\langle .,.\rangle $ be the Euclidean norm and
the scalar product in $\R^n$, respectively. For a function $u(t,x)$ defined on $\Omega,$ we denote $ D_xu=(u_{x_1},\dots, u_{x_n})$ and $Du=(u_t, D_xu).$ The Hessian of a function $h(x)$ on $\R^n$ will be denoted by $D^2h(x).$

\section{Preliminaries and Hopf-Lax formulas}
\subsection{Assumptions}

\medskip
 
Consider the following problem $(H,\sigma)$
\begin{equation}\label{2.1}\frac{\partial u}{\partial t} + H(D_x u)=0\, , \,\,
(t,x)\in \Omega =(0,T)\times \R^n,\end{equation}
\begin{equation}\label{2.2}u(0,x)=\sigma(x)\, , \,\, x\in
\R^n.\end{equation}

We suppose that the following conditions hold for $H(p)$
and $\sigma (x)$.

(H0) $H(p)$ is strictly convex on $\R^n$ and superlinear, i.e., 
$\lim\limits_{\|p\|\to\infty}\displaystyle\frac{H(p)}{\|
p\|}=+\infty.$

(H1) $\sigma (x)$ is continuous on $\R^n$ and for every
$(t_0,x_0)\in [0,T)\times \R^n$ there exist positive constants
$r,\ N$ such that

\begin{equation}\label{2.3}\min_{\|y\|\le N} \Big\{ \sigma (y)+tH^*\big (\frac
{x-y}{t}\big)\Big \} \ <\  \sigma (z)+tH^*\big (\frac
{x-z}{t}\big) \end{equation} as $\|z\| >N,$ $|t-t_0|+\|x-x_0\|<r,\
t\ne 0,$  where $H^*$ is the Fenchel conjugate function of $H,$ i.e., 
$$H^*(z)=\max_{y\in \R^n}\{\langle z,y\rangle -H(y)\},\ \forall  z\in \R^n.$$

\smallskip
  Denote
$$\zeta (t,x,y)  =\sigma (y)+tH^*\big (\frac
{x-y}{t}\big),\ (t,x,y)\in (0,T)\times\R^n\times\R^n, $$ and
\begin{equation}\label{l} \ell (t,x)=\{y_0\in \R^n\ |\ \zeta(t,x,y_0)=\min_{\|y\|\le N} \zeta(t,x,y)=\min_{\R^n} \zeta(t,x,y)\},\end{equation}
then $\ell (t,x)$ is a compact set in $\R^n$ by the hypothesis (H1)
and the lower semi-continuity of the function $\zeta (t,x,.).$

\begin{rem}

1. The assumption (H1) can be considered as a compatible condition for the Hamiltonian $H(p)$ and initial data $\sigma
(x).$ 
\smallskip

2.  Assume (H0) holds and $\sigma (x)$ is Lipschitz continuous on $\R^n.$  Then condition (H1) is automatically satisfied. In the sequel, dependent on the context, some time we use assumption (H0), (H1) or (H0) and $\sigma(x)$ is Lipschitz continuous.

\end{rem}

\begin{defn}
Assume (H0), (H1). Then the function $u(t,x)$ defined on $\Omega$ by the formula
\begin{equation}\label{hl1}u(t,x)=\min_{y\in \R^n} \Big\{ \sigma (y)+tH^*\big (\frac
{x-y}{t}\big)\Big \}\end{equation}
is called Hopf-Lax formula of the problem (2.1)-(2.2).
\end{defn}
\begin{rem}
Analogously, we can establish the Hopf-Lax formula for concave Hamiltonian as follows. Let $K:\, \R^n\to\R$ be a concave function. Assume that the convex function $H(p)=-K(p)$ and $-\sigma(x)$ satisfy the condition (H0), (H1). Then the Hopf-Lax formula for the Cauchy problem $(K,\sigma)$ can be represented as
\begin{equation}\label{hl2}u(t,x)=\max_{y\in \R^n} \Big\{ \sigma (y)-t(-K)^*\big (\frac
{y-x}{t}\big)\Big \}.\end{equation}

\end{rem}

\subsection{Sub- and superdifferential of a function.}

Let  $\Om$ be and open set of $ \R^m$ and let $v:\ \Om \to \R$ be a function and $y\in \Om.$ We denote by $D^+v(y)$ (resp. $D^-v(y))$ the set of all $\textbf{p}\in \R^{m}$ defined by:
$$D^+v(y)=\{ \textbf{p}\in \R^{m}\, |\ \lim\sup_{h\to 0} \frac{v(y+h)-v(y)-\langle \textbf{p},h\rangle }{|h|}\ \le \ 0\} $$
(resp. 
$$D^-v(y)=\{\textbf{p}\in \R^{m}\, | \ \lim\inf_{h\to 0} \frac{v(y+h)-v(y)-\langle \textbf{p},h\rangle }{|h|}\ \ge \ 0\}).$$

We call $D^+v(y)$ (resp. $D^-v(y))$ superdifferential (resp. subdifferential) of the function $v(y)$ at $y\in \Om.$ In general, $D^+v(y)$ or $D^-v(y)$ is called  the semidifferential of $v$ at $y\in \Om.$
\smallskip

Note that, if $v$ is a convex function, then the subdifferential $D^-v(x)$ coincides with subgradient $\partial v(x)$ of the function $v.$

The following proposition presents some elementary properties of semidifferentials of  functions that are necessary in the sequel.
\begin {prop}\label{p1}
Let $u, v$ be  functions defined on an open set $\Om \subset \R^m.$

(i)  Suppose that the function $v$ attains a minimum (resp. maximum) at $y_0\in \Om.$ Then $0\in D^-v(y_0)$ (resp. $ 0\in D^+v(y_0)).$
\smallskip

(ii) If the semidifferentials of $u$ and $v$ at $y$ are not empty, then $D^{\pm} u(y) +D^{\pm} v (y) \subset D^{\pm}(u+v)(y).$ 
\smallskip

(iii) Suppose that, at a point $y\in \Om,$  one of two functions $u, \; v$  is differentiable and the semidifferential of the other is not empty. Then we have
$D^{\pm} u(y) +D^{\pm} v(y) = D^{\pm}(u+v)(y).$
\end{prop}

\begin{proof}
(i) Suppose that $v$ attains a minimum at $y_0.$ Then for any small enough $h,$ we have $v(y+h)\ge v(y_0).$ Then
$$\lim\inf_{h\to 0} \frac{v(y+h)-v(y_0)}{|h|}\ \ge \ 0.$$ Thus, $0\in D^-v(y_0).$
\smallskip

(ii) We check for the case of subdifferential since the other is the same. Take $p\in D^-u(y),\ q\in D^-v(y),$  then
\begin{equation} \aligned
\liminf_{h\to 0}&\frac{(u+v)(y+h)-(u+v)(y) -\langle p+q,h\rangle}{|h|}\ge\\
 \liminf_{h\to 0}&\frac{(u)(y+h)-(u)(y) -\langle p,h\rangle}{|h|}
+\liminf_{h\to 0}\frac{(v)(y+h)-(v)(y) -\langle q,h\rangle}{|h|} \ge 0.\endaligned\end{equation}

Thus, $D^{-} u (y)+D^{-} v(y) \subset D^{-}(u+v)(y).$
\smallskip

(iii)  Now assume that $v$ is differentiable at $y\in \Om$ and $D^-u(y)\ne \emptyset, $ then $D^-(u+v)(y)\ne \emptyset.$ Take $s\in D^-(u+v)(y)$ and set $p=s-q$ where $q=Dv(y),$ we have
$$\aligned \frac{u(y+h)-u(y) -\langle p,h\rangle}{|h|}&=\frac{(u+v)(y+h)-(u+v)(y) -\langle p+q,h\rangle}{|h|}\\
& -\frac{v(y+h)-v(y) -\langle q,h\rangle}{|h|}.\endaligned$$
Taking $\liminf$ of both sides as $h\to 0$ and noting that the last expression tends to $0,$  we see that $p\in D^-u(y).$ Thus $s=p+q,$ and we get that $D^-(u+v)(y)=D^-u(y) +\{Dv(y)\}.$
\end{proof}

\begin{defn}
Let $\mathcal O$ be a convex subset of $\R^m$ and let $v:\ \mathcal O \to \R$ be a continuous function.

(a) The function $v$ is called {\it semiconcave} with linear modulus if there is a constant $C>0$ such that
$$\lambda v(y_1)+(1-\lambda ) v(y_2)-v(\lambda y_1+(1-\lambda)y_2) \le \lambda (1-\lambda)\frac C2 |y_1-y_2|^2$$
for any $y_1,\, y_2$ in $\mathcal O$ and for any $\lambda \in [0,1].$ Then the number $C$ is called a {\it semiconcavity constant} of $v.$ The function $v$ is {\it semiconvex} if and only if $-v$ is semiconcave.

(b) The function $v$ is called {\it uniformly convex} with constant $\Lambda>0$ if $v(y)-\Lambda|y|^2,\ y\in \mathcal O$ is a convex function.con
\end{defn}

\begin{rem}
1. The theory of semiconcave functions has been fully studied since
the last decades of previous century. The reader is referred to the monograph \cite{cs} for a comprehensive development of the topic.
\smallskip

2. The notion of semiconcavity (resp. uniform convexity) is a special case of the notion $\sigma$-smoothness (resp. $\rho$-convexity) of a function, see \cite{ap}. The following proposition is extracted from Prop. 2.6 of the just cited article.
\end{rem}

\begin{prop}\label{scc}Let $v:\R^m \to \R$ be a convex function. Moreover, 

(i) Suppose that $v$ is uniformly convex with a constant $C>0.$ Then the Fenchel conjugate function $v^*$ is a semiconcave function with a semiconcavity constant $\frac1C >0.$

(ii) Suppose that $v$ is a semiconcave function with a semiconcavity constant $C^*>0.$ Then $v^*$ is a uniformly convex function with a constant $\frac1{C^*}.$

\end{prop}
\begin{defn}

A continuous functions $u: [0, T]\times \R^n\to \R$ is called a {\it viscosity solution} of the Cauchy problem (1.1)-(1.2) on $\Omega=(0,T)\times \R^n,$ provided that the following hold:\\[1ex]
{\rm (i)} $ \lim_{(t,x)\to (0,x)}u(t,x)=\sigma(x)$ for all $x\in \R^n$;\\
{\rm (ii)} For each $(t_0,x_0)\in  \Omega$ and $(p,q)\in D^+u(t_0,x_0),$ then
\begin{equation} \label{sub}p +H(t_0,x_0,q) \le 0,\end{equation}
and for each $(t_0,x_0)\in  \Omega$ and $(p,q)\in D^-u(t_0,x_0),$ then
\begin{equation} \label{sup}p +H(t_0,x_0,q) \ge 0.\end{equation}

If the continuous function $u$ satisfies (i) and (ii)-(\ref{sub}) (resp. (ii)-(\ref{sup})), then it is called a viscosity subsolution (resp. supersolution) of the problem (1.1)-(1.2). It is noted that, there are several propositions which are equivalent to this definition, e.g., the notion of $C^1$-test function is used instead of semidifferentials, see \cite{cl}.
\end{defn}

\medskip

We collect here some important properties of Hopf-Lax formula, given by the following theorem. The proof is as the one of cited papers with minor adjustment.

\begin{thm} (see \cite {ev,cs,nh1}). Assume (H0), (H1). Then

\smallskip
(a)  The function $u(t,x)$ defined by Hopf-Lax formula
$$u(t,x)=\min_{y\in \R^n} \Big\{ \sigma (y)+tH^*\big (\frac
{x-y}{t}\big)\Big \}$$
 is a viscosity solution of
the problem (\ref{2.1})-(\ref{2.2}). Moreover, it is locally Lipschitz continuous on $\Omega.$

(b) Suppose that $\sigma(x)$ is semiconcave on $\R^n$ or $H$ is uniformly convex. Then for all $t\in (0,T],$ the solution $u(t,x)$ is semiconcave in $x$ on $\R^n.$

(c) The solution $u(t,x)$ is continuously differentiable in
some open $V\subset \Omega$ if and only if $ \ell(t,x)$  is a
singleton for all $(t,x)\in V.$

\end{thm}

\begin{rem} From (c), we note that, if  $u(t,x)$ is differentiable at $(t_0,x_0)$ then $\ell (t_0,x_0)=\{p\}$  is a singleton; see \cite{nh1}. On the other hand, the sufficient condition in b) will be improved in the theorem below.
\end{rem}

 \begin{thm}\label{di} Assume (H0), (H1). Let $(t_0,x_0)\in \Omega$ such that $\ell(t_0,x_0)$ is a singleton. Then Hopf-Lax formula $u(t,x)$ defined by (\ref{hl1}) is differentiable at $(t_0, x_0)$.
 \end{thm}

\begin{proof}
Suppose that $\ell (t_0,x_0)=\{y_0\}.$ We denote
$$p_t =H^*(\frac {x_0-y_0}{t_0})-\frac 1{t_0}\langle x_0-y_0,H^*_z(\frac{x_0-y_0}{t_0})\rangle;\ p=H^*_z(\frac{x_0-y_0}{t_0})$$
and let
$$\alpha =\liminf_{(h,k)\to (0,0)}\frac{u(t_0+h,x_0+k)-u(t_0,x_0) -p_t h -\langle p,k\rangle}{\sqrt{h^2+|k|^2}},$$
for  $(h,k)\in \R\times \R^n.$ 

Then there exists a sequence $(h_m,k_m)\to (0,0)$ such that $\lim_{m\to \infty} \Phi_m =\alpha,$ where
$$ \Phi_m = \frac{u(t_0+h_m,x_0+k_m)-u(t_0,x_0) -p_t h_m -\langle p,k_m\rangle}{\sqrt{h_m^2+|k_m|^2}}.$$

For each $m\in \N,$ we take $y_m\in \ell(t_0+h_m,x_0+k_m)$ then
\begin{equation}\label{ps}\aligned \Phi_m\ \ge  &\ \frac{\varphi (t_0+h_m,x_0+k_m,y_m)-\varphi (t_0,x_0,y_m)-p_th_m -\langle p,k_m\rangle}{\sqrt{h_m^2+|k_m|^2}}\\
\ge  &\  \frac{(t_0+h_m)H^*(\frac{x_0+k_m-y_m}{t_0+h_m}) -t_0H^*(\frac {x_0-y_m}{t_0})-p_t h_m -\langle p,k_m\rangle}{\sqrt{h_m^2+|k_m|^2}},\endaligned\end{equation}
Applying the mean value theorem for the functions $(t,x)\mapsto tH^*(\frac{x-y_m}{t})$ on the line segment $[(t_0,x_0-y_m), (t_0+h_m, x_0+k_m-y_m)]$ we can see that the numerator of right hand side of (\ref{ps}) is equal to the following expression:
 $$\aligned
 \Psi_m &=h_m\Big (H^*(\frac{x_0+k_m^* -y_m}{t_0+h_m^*})-\frac 1{t_0+h_m^*}\langle H^*_z(\frac{x_0+k_m^*
 -y_m}{t_0+h_m^*}),x_0+k_m^*-y_m\rangle-p_t\Big ) \\
&+k_m\Big (H^*_z(\frac{x_0+k_m^*-y_m}{t_0+h_m^*})-p\Big),\endaligned$$
where $|h_m^*|\le |h_m|,\ |k_m^*|\le |k_m|.$
\smallskip

Taking into account the assumption (H1), it is easy to see that, for
$(h_m, k_m)$ small enough, the sequence $(y_m)_m$ is bounded, then
we can choose a subsequence also denoted by $(y_m)_m$ such that
$y_m\to y^*$ as $m\to \infty.$ Since the set-valued mapping
$(t,x)\mapsto \ell(t,x)$ is upper semicontinuous \cite{vhs}, then
$y^*\in \ell(t_0,x_0),$ that is $y_0=y^*.$

Now, letting $m\to \infty$ we have
$$\alpha=\lim_{m\to\infty}\, \Phi_m \ge \lim_{m\to\infty} \frac{\Psi_m}{\sqrt{h_m^2+|k_m|^2}} =0.$$

On the other hand, let
$$\beta =\limsup_{(h,k)\to (0,0)}\frac{u(t_0+h,x_0+k)-u(t_0,x_0) -p_t h -\langle p,k\rangle}{\sqrt{h^2+|k|^2}}. $$

We have, for $y_0\in \ell (t_0,x_0), $ then
$$u(t_0,x_0) =\sigma(y_0) +t_0H^*(\frac{x_0-y_0}{t_0}) , \ \text{and}$$  
$$u(t_0+h, x_0+k)\le \sigma(y_0) +(t_0+h)H^*(\frac{x_0+k-y_0}{t_0+h}).$$
 Thus, applying the mean value theorem as above, we have
$$\aligned &u(t_0+h,x_0+k)-u(t_0,x_0) \le (t_0+h)H^*(\frac{x_0+k-y_0}{t_0+h}) -t_0H^*(\frac{x_0-y_0}{t_0}) \\
&\le h\Big (H^*_z(\frac{x_0+k^*-y_0}{t_0+h^*}) -\frac{1}{t_0+h}\langle x_0+k^*-y_0,H^*_z(\frac{x_0+k^*-y_0}{t_0+h^*}) 
-p_t\Big) \\ &+k\Big (H^*_z(\frac{x_0+k^*-y_0}{t_0+h^*})-p\Big )=\Psi(h,k),\endaligned$$
where $|h^*|\le |h|, \, |k^*|\le |k|.$  Therefore
$$\beta \le \limsup_{(h,k)\to (0,0)}\frac{\Psi(h,k)}{\sqrt{h^2+|k|^2}} =0.$$

Thus,
$$\lim_{(h,k)\to (0,0)}\frac{u(t_0+h,x_0+k)-u(t_0,x_0) -p_t h -\langle p,k\rangle}{\sqrt{h^2+|k|^2}} =0.$$

The theorem is then proved.
\end{proof}

\begin{rem}
We suppose that for fixed $t_0\in (0,T],$ the function $u(t,\cdot)$ is differentiable at $x_0\in \R^n.$ Then $\ell(t_0,x_0)$ is a singleton, see \cite{vhs}. By above theorem, we deduce that $u(t,x)$ is also differentiable at $(t_0,x_0)$ as a function of two variables.
\end{rem}

\begin{thm}\label{sc} Assume  (H0), (H1). In addition, let $H=H(p)$ be a semiconcave function with the semiconcavity constant $\theta^{-1} >0$  and let $\sigma$ be a semiconvex function with constant $B>0.$ Then there exists $t_*\in (0, T)$ such that for all $t_0\in (0, t_*),$ the function $v(x)=u(t_0,x)$ is semiconvex, where $u(t,x)$ is the Hopf-Lax formula  defined by (\ref{hl1}).
 \end{thm}

\begin{proof}By assumption and Prop. \ref{scc}, we first note that the Legendre conjugate function $H^*$ is a uniformly convex function with constant $\theta>0.$ Therefore the function $\phi (p)=H^*(p) -\frac{\theta}2 |p|^2$ is convex and then, for all $a,b\in \R^n$  we have 
\begin{equation}\label{sc1} H^*(a) +H^*(b) -2H^*(\frac{a+b}{2}) \ge \frac{\theta}{2} (|a|^2+|b|^2-2|\frac{a+b}{2}|^2)=\frac{\theta}4|a-b|^2.
\end{equation}
Next, we follow the argument in the proof of Theorem 3.5.3 (iv) \cite{cs} with an appropriate adjustment. Take $\gamma >0$ such that $\Lambda =\theta \gamma>B$ and then choose $t_*\in (0,T)$ such that $0<\gamma t_*\le 1.$ Let $ t_0\in (0,t_*), x_1, x_2\in \R^n,$ pick out $y_1\in \ell(t_0,x_1), y_2\in \ell(t_0,x_2);$ using the inequality (\ref{sc1}) we have
$$\aligned & u(t_0,x_1)+u(t_0,x_2)-2u(t_0,\frac{x_1+x_2}2)\ge \sigma(y_1)+t_0H^*(\frac{x_1-y_1}{t_0}) +\sigma(y_2)\\
&+t_0H^*(\frac{x_2-y_2}{t_0})-2\sigma(\frac{y_1+y_2}2)-2t_0H^*(\frac{x_1-y_1+x_2-y_2}{2t_0})\ge \sigma(y_1)\\
&+\sigma(y_2)-2\sigma(\frac{y_1+y_2}2) +t_0\big (H^*(\frac{x_1-y_1}{t_0}) +H^*(\frac{x_2-y_2}{t_0})-2H^*(\frac{x_1-y_1+x_2-y_2}{2t_0})\big)\\
&\ge -B|y_1-y_2|^2 +\frac{\theta}{4t_0}|(x_1-x_2)-(y_1-y_2)|^2\\
&\ge -B|y_1-y_2|^2 +\frac 1{\gamma t_0}\; \frac{\Lambda}4\big(|x_1-x_2|^2 +|y_1-y_2|^2-2\langle x_2-x_1, y_2-y_1\rangle\big).\endaligned $$

Using the obvious inequality $2\langle a,b\rangle \le \frac{|a|^2}{\epsilon} +\epsilon|b|^2$ for $\epsilon >0,$ we see that
$$\aligned &\frac{\Lambda}4\big(|x_1-x_2|^2 +|y_1-y_2|^2-2\langle x_2-x_1, y_2-y_1\rangle\big)\\
&\ge \frac{\Lambda}4\big(|x_1-x_2|^2 +|y_1-y_2|^2- \frac{\Lambda}{\Lambda -B}| x_1-x_2|^2-\frac{\Lambda -B}{\Lambda}|y_1-y_2|^2\big)\\
&=-\frac{\Lambda B}{\Lambda-B}\big|\frac{x_1-x_2}2\big|^2 +B\big|\frac{y_1-y_2}2\big|^2.\endaligned$$

Therefore,
$$\aligned & u(t_0,x_1) +u(t_0,x_2)-2u(t_0,\frac{x_1+x_2}2)\\
&\ge -\frac{\Lambda B}{\gamma t_0(\Lambda -B)}\big|\frac{x_1-x_2}2\big|^2 +(\frac B{\gamma t_0}-B)|y_1-y_2|^2\\
&\ge -\frac{\Lambda B}{\gamma t_0(\Lambda -B)}\big|\frac{x_1-x_2}2\big|^2.\endaligned$$

Thus, the function $v(x)=u(t_0,x)$ is a semiconvex function.
\end{proof}

\section{Generalized characteristics}

In this section we focus on the study of the differentiability of function $u(t,x)$ given by Hopf-Lax formula on the characteristics. To this aim, let us recall the Cauchy method of characteristics for problem (\ref{2.1})-(\ref{2.2}). 
\medskip

We first suppose that $H(p)$ and $\sigma (x)$ are of class $C^2.$
\medskip

The characteristic differential equations of problem
(\ref{2.1})-(\ref{2.2}) is as follows
\begin{equation}\label{2.5}
\dot x=H_p \ ;\qquad \dot v = \ \langle H_p,p\rangle
 - \ H \ ;\qquad \dot p=0 \,\end{equation}
with initial conditions
\begin{equation}\label{2.6} x(0,y)=y \ ;\qquad v(0,y)=\sigma(y)\ ;\qquad p(0,y)=
\sigma _y(y)\ ,\quad y\in \R^n.\end{equation}

Then a characteristic strip of the problem
(\ref{2.1})-(\ref{2.2}) (i.e., a solution of the system of
differential equations (\ref{2.5}) - (\ref{2.6})) is defined by
\begin{equation}\label{2.7}\simuleq{x&=x(t,y)=y+t H_p(\sigma_y(y), \\
v&=v(t,y)=\sigma(y)+t\big (
\langle H_p(\sigma_y(y)),\sigma_y(y)\rangle-H(\sigma_y(y))\big),\\
 p&= p(t,y)\ =\ \sigma_y(y).}\end{equation}

The first component of solutions (\ref{2.7}) is called the characteristic curve (briefly, characteristics) emanating from $(0,y), y\in \R^n,$ i.e., the curve defined by
\begin{equation}\label{2.8}\mathcal C:\ x=x(t,y)=y+t H_p(\sigma_y(y)),\ t\in [0,T]. \end{equation}

Let $t_0\in (0,T].$ If for any $t\in (0,t_0)$ such that $x(t,\cdot):\ \R^n \to \R^n$ is a diffeomorphism, then $u(t,x)=v(t, x^{-1}(t,x))$ is a $C^2$ solution of the problem on the region $(0,t_0)\times \R^n.$

\medskip
Now we assume that $\sigma$ is merely a continuous function on $\R^n.$ We use following notation
$$D^\# \sigma(y)=\simuleq{& D^+\sigma(y) \cup D^-\sigma(y), \ {\text {if}}\  D^+\sigma(y) \cup D^-\sigma(y)\ne \emptyset, \\
&\{0\}, \ {\text {otherwise.}} }$$

It is known that, the subset of $\R^n$ such that $D^+\sigma (y)\ne \emptyset$ (resp. $D^-\sigma (y)\ne \emptyset)$ is dense in $\R^n.$ Now, the initial condition $p(0,y)=\sigma_y(y)$ can be replaced as follows
$$p(0,y)=q,\ \text {for some}\ q \in D^\#\sigma(y),$$
then a {\it generalized characteristics} emanating from $(0,y)$ is defined by the curve $$\mathcal C:\  x=x(t,y)=y+tH_p(q), \ q\in D^\#\sigma(y).$$

For each  $(t,x)\in (0,T]\times\R^n,$ we denote by $\ell^*(t,x)$ the set of  $y\in \R^n$ such that there is a generalized characteristic curve starting from
 $(0,y), y\in \R^n$ goes through the point $(t,x).$ In other words, $y\in \ell^*(t,x)$ if and only if
$x=y+tH_p(q),$ for some $q\in D^\#\sigma(y).$

\smallskip
Using the initial condition $q\in D^\#\sigma(y)$ for $p(0,y)$, we see that in general, there is a bundle of generalized characteristics that emanates at an initial point $(0,y), y\in \R^n.$ Nevertheless, in the following, we can single out ``right" characteristics going through a point $(t_0,x_0),\ t_0>0$ as in the case  the function $\sigma $ is of class $C^1.$
\smallskip

The following theorem is known under assumptions that $H,\sigma$ are of class $C^2,$ e.g., see \cite{cs,nh1}. It remains true for the case
 of generalized characteristic curves and the proof is similar to the earlier one. For the reader's convenience, we write down the argument.

\begin{thm} Let  $H \in C^1(\R^n)$ and $\sigma\in C(\R^n)$ satisfying (H0), (H1) and let $(t,x)\in (0,T]\times \R^n.$
Then $\ell(t,x)\subset
\ell^*(t,x)$ and the solution $u(t,x)$ of the problem (\ref{2.1})- (\ref{2.2}) given by
Hopf-Lax formula (\ref{hl1}) can be defined as follows
$$u(t,x)= \min_{y\in \R^n} \Big\{ \sigma (y)+tH^*\big (\frac
{x-y}{t}\big)\Big \}=\min_{y\in \ell^*(t,x)} \Big\{ \sigma
(y)+tH^*\big (\frac {x-y}{t}\big)\Big \}.$$ 
\end{thm}

\medskip
\begin{proof} Let
$\zeta (t,x,y)=\sigma (y)+tH^*\big(\frac{x-y}{t}\big),\quad (t,x)\in \Omega,\ y\in \R^n.$
Then $$u(t,x)\ =\ \min_{y\in \R^n}\, \zeta (t,x,y).$$ 

Since $H$ is a strictly convex and superlinear function, then $H^*$
is continuously differentiable on $\R^n.$ Therefore the minimum of
$\zeta (t,x,.)$ is attained at some   $y\in \R^n$ which is a
stationary point of $\zeta(t,x,.). $ Then by Propositions \ref {p1} we get
$$0\in D^{-}_y\zeta (t,x,y)= D^-_y\sigma(y)-H^*_z\big(\frac{x-y}{t}\big).$$ 
Thus, $H^*_z\big(\frac{x-y}{t}\big) =q\in D^-\sigma(y).$
From this equality and by a differential property of the Fenchel conjugate
functions, we have
  $$H_p(q)=\displaystyle\frac{x-y}{t}.$$
Consequently,   $x=y+tH_p(q),$ where $q\in D^-\sigma(y),$ thus $y\in \ell^*(t,x)$. Since
then, we have
$$\min_{y\in \R^n} \Big\{ \sigma (y)+tH^*\big (\frac
{x-y}{t}\big)\Big \}=\min_{y\in \ell^*(t,x)} \Big\{ \sigma
(y)+tH^*\big (\frac {x-y}{t}\big)\Big \}$$ and \ $\ell(t,x)\subset
\ell^*(t,x).$
\end{proof}

\begin{cor} Let $(t_0,x_0)\in (0,T]\times \R^n.$ Then for each $y\in \ell(t_0,x_0),$  there is a unique characteristic curve of the form $\mathcal C:\ x=x(t,y)=y+tH_p(q)$ starting at $(0,y)$ and goes through the point $(t_0,x_0),$ where $q=H^*_z\big(\frac{x_0-y}{t_0}\big) \in D^-\sigma(y).$
\end{cor}

\begin{proof} Since $\emptyset \ne \ell(t_0,x_0)\subset \ell^*(t_0,x_0),$ we take $y\in \ell(t_0,x_0),$ then $H^*_z\big(\frac{x_0-y}{t_0}\big) =q\in D^-\sigma(y).$ Let $\mathcal C :\ x=x(t,y)=y+tH_p(q)$ be the corresponding generalized characteristic curve defined by $y$ and $q,$ we rewrite $\mathcal C$ as $x=x(t,y) =y+t(\frac{x_0-y}{t_0}).$ Thus it is obvious that $\mathcal C$ goes through the point $(t_0,x_0)$  and it is unique since $\mathcal C$ is the straight line joining two points  $(0,y)$ and $(t_0,x_0).$
\end{proof}

Now, let $\mathcal C: \ x=x(t,y)=y+t H_p(q),\ q\in D^\#\sigma(y)$  be a generalized characteristic curve starting from $y\in \R^n.$ Suppose that $\mathcal C$ goes through a point $(t_0,x_0), \;t_0>0 $ then we have $x_0=y+t_0H_p(q).$ Thus, $\mathcal C$ can be rewritten as the form
$$x=x(t)=x_0+(t-t_0)H_p(q)$$
where $q\in D^\#\sigma(y)$ and $y$ verifies the equality $y=x_0-t_0H_p(q),$  i.e., $y\in \ell^*(t_0,x_0).$ From Corollary 3.2, we get the following definition.

\begin{defn}
We say that the characteristic curve $\mathcal C: \ x=x(t,y)=y+t H_p(q),\ q\in D^\#\sigma(y)$ is of the {\it type}  (I)  at point $(t_0,x_0) \in (0,T]\times\R^n$, if $y\in \ell(t_0,x_0).$ If $y\in \ell^*(t_0,x_0)\setminus \ell(t_0,x_0)$ then $\mathcal C$ is said to be the {\it type}  (II)  at $(t_0,x_0).$
\end{defn}

To analyze properties of type of characteristic curves at a point $(t_0,x_0) \in (0,T]\times\R^n$ we need the following lemmas.

\begin{lem} \label{ev}\cite{ev}
For each $x\in \R^n$ and $0\le s<t\le T,$ the Hopf-Lax formula (\ref{hl1}) can be represented as
$$u(t,x)=\min_{y\in \R^n}\Big\{ u(s,y) +(t-s)H^*\Big(\frac{x-y}{t-s}\Big) \Big \}.$$

\end{lem}

A useful property of  strictly convex function can be deduced from the following.
\begin{lem} \label{h3}
Let $v$ be a convex function and $D=\textrm{dom}\; v\subset \R^n.$ Suppose that there exist $p,\,  p_0\in D,\ p\ne p_0$ and $y\in \partial v(p_0)$ such that 
$$\langle y,p-p_0\rangle = v(p) -v(p_0).$$

Then for all $z$ in the straight line segment $[p,p_0]$ we have 
$$v(z) =\langle y,z\rangle -\langle y,p_0\rangle +v(p_0).$$

Moreover, $y\in \partial v(z)$ for all $z\in [p,p_0].$
\end{lem}
\begin{proof} See \cite{nh3}.
\end{proof}

Now we investigate  properties of characteristic curves of type (I), (II) at $(t_0,x_0)$ given by the following theorems.

\medskip

\begin{thm}\label{cha1} Assume (H0), (H1).  Let $(t_0,x_0)\in (0,T] \times \R^n,\ q=H^*_z(\frac{x_0-y_0}{t_0})\in D^-\sigma(y_0)$ where $y_0\in \ell(t_0,x_0)$ and let 
$$ \mathcal C: x= x(t)= x_0 +(t-t_0)H_p(q), $$
be a generalized characteristic curve of type (I) at $(t_0,x_0).$  Then $y_0\in \ell (t,x)$ for all $(t,x)\in \mathcal C, \ 0< t\le t_0.$ Moreover, $\ell (t,x) =\{y_0\}$ for all $(t,x)\in \mathcal C, \ 0\le t < t_0.$ 
\end{thm}

\begin{proof} 
Let $(t_1,x_1)\in \mathcal C,\ 0< t_1\le t_0.$ By Hopf-Lax formula, one has
$$u(t_0,x_0)=\sigma(y_0) +t_0H^*(\frac{x_0-y_0}{t_0}),\ y_0\in \ell(t_0,x_0).$$
By Lemma \ref{ev}, $u(t_0,x_0)=\min_{z\in\R^n}\{u(t_1,z)+(t_0-t_1)H^*(\frac{x_0-z}{t_0-t_1})\}$ thus,
$$u(t_0,x_0)\le u(t_1,x_1)+(t_0-t_1)H^*(\frac{x_0-x_1}{t_0-t_1}).$$ 
Therefore,
$$\sigma(y_0)+t_0H^*(\frac{x_0-y_0}{t_0})\le u(t_1,x_1)+(t_0-t_1)H^*(\frac{x_0-x_1}{t_0-t_1}).$$
Since the points $(0, y_0),\, (t_0,x_0),\, (t_1,x_1)$ belong to $\mathcal C,$ we have
$$\frac{x_1-x_0}{t_1-t_0} =\frac{x_1-y_0}{t_1}=\frac{x_0-y_0}{t_0}=H_p(q),$$ then 
$$\sigma(y_0) +t_1H^*(\frac{x_1-y_0}{t_1})\le u(t_1,x_1).$$ Thus, $y_0\in \ell(t_1,x_1).$

\smallskip

Next, we prove that $\ell(t,x)=\{y_0\}$ for all $(t,x)\in \mathcal C, \; 0<t<t_0.$ To this end, take an arbitrary $y\in \R^n, y\ne y_0$ and denote by 
$$ \eta(t,y) =\zeta (t,x,y)-\zeta (t,x,y_0),\ (t,x)\in \mathcal C, \ t\in (0,t_0],$$
where $\zeta (t,x,y)=\sigma(y)-tH^*(\frac{x-y}{t}).$ Then
\begin{equation}\label{2.13} \eta(t,y)=\sigma(y)-\sigma(y_0) +t\big(H^*(\frac{x-y}{t}) -  H^*(\frac{x-y_0}{t})\big ).\end{equation} 
for $t\in (0, t_0], \ x=x(t).$ 
\smallskip

We shall prove that $\eta(t,y)>0$  for all $t\in (0,t_0).$
\smallskip

It is obviously that, $\eta(t_0,y)\ge 0.$ 
Note that $x=x(t)=x_0+(t-t_0)H_p(q),$ thus $y_0=x(0)=x_0-t_0H_p(q).$  Then from (\ref{2.13}) with a simple calculation, we obtain
$$\eta'_t(t,y)=H^*(\frac{x-y}{t})-H^*(\frac{x-y_0}{t}) +\frac 1t\Big\langle y-y_0,H^*_z(\frac{x-y}{t}) \Big\rangle.$$

Since $H^*$ is a strictly convex function and $y\ne y_0,$ we use a monotone property of subdifferential of convex function and Lemma \ref{h3} to get
$$H^*(\frac{x-y}{t}) -  H^*(\frac{x-y_0}{t}) < \frac 1t\Big \langle {y_0-y}, H^*_z(\frac {x-y}{t})\Big \rangle .$$
Thus, $\eta'_t(t,y)<0,\ \forall t\in (0,t_0)$ and the function $\eta(t,y)$ is  strictly decreasing on the interval $[0, t_0].$ Therefore, 
$$\eta(t,y)>\eta(t_0,y)\ge 0,\ \forall t\in [0,t_0).$$ This means that, if $y\ne y_0$ then $y\notin \ell(t_0,x_0).$ The theorem now is proved.
\end{proof}

{\bf Example.} Consider the following Cauchy problems

$$\begin{cases}\frac{\partial u}{\partial t} + \frac 12 |u_x|^2&=0\, , \,\,
(t,x)\in \Omega =(0,T)\times \R,\\
u(0,x)&=\sigma(x)\, , \,\, x\in
\R.
\end{cases}$$

{\bf (i)}  Let $\sigma(x) =-|x|.$ Then Hopf formula gives viscosity
$$u(t,x)=\min_{y\in \R} \{-|y|+\frac 1{2t} (x-y)^2\}= -|x|-\frac t2.$$
 For $(t_0,x_0)=(1,0),$ we see that $\ell(1,0)=\{1,-1\}$ and $D^-\sigma(0) =\emptyset, D^+\sigma(0) =[-1,1].$ The characteristic curves $\mathcal C$ of the problem have the form $x=x(t,y)=y+tq,\ q\in D^\#\sigma(y).$ Suppose that $\mathcal C$ goes through $(1,0),$  then $0=y+q.$ If $\mathcal C$ starts from $(0,0)$ then $q=0\in D^\#\sigma(0),$ otherwise, where $y\ne 0,$ then $q=\pm 1,$ thus $y=\mp 1.$ Thus, the characteristic curves going through $(1,0)$ consist of two of type (I): $x=(1,t) =1-t,\  x=x(-1,t)=-1+t$ and one of type (II) is $x= x(0,t)=0.$ 
\smallskip
Note that a bundle of characteristic curves emanating from $(0,0)$ is 
$ x=x(t,0)=\alpha t,$ where $\alpha \in [-1,1].$
\smallskip

{\bf (ii)} Let $\sigma(x) =-|x|.$ Then Hopf formula gives a viscosity 
$$u(t,x)=\begin{cases} x-\frac 12, & x>t\\
-x-\frac t2, & x<t\\
\frac{x^2}{2t}, & |x|\le t,\end{cases}$$
and $u(t,x)$ is of class $C^1(\Omega).$ 
For any $(t_0,x_0)\in \Omega$ there is a unique characteristic curve going through this point, since $u(t,x)$ is differentiable at there. 
\smallskip

In this case, a bundle of characteristic curves emanating from $(0,0)$ is also
$ x=x(t,0)=\alpha t,$ where $\alpha \in [-1,1].$ All characteristic curves in the bundle are of type (II), except $x=x(t,0)=0$ is of type (I).
\medskip

By above examples, we see that, if the characteristic curve $\mathcal C$ is of type (II) at $(t_0,x_0)$ then it may be of type (II) at any point $(t,x)\in \mathcal C, \ 0\le t\le t_0.$ Nevertheless, if the given data are of class $C^2(\R^n),$ we have the following:

\begin{thm} \label{cha2} Assume (H0), (H1).  In addition, suppose that $H,\ \sigma$ are of class $C^2.$  Take $(t_0,x_0)\in \Omega$ and let $\mathcal C: x =x(t) =x_0 +t H_p(\sigma_y(y_0)) $ be a characteristic curve of type (II) at $(t_0,x_0).$   Then there exists $\theta \in (0,t_0)$ such that $\mathcal C$ is of type (I) at $(\theta, x(\theta))$ and $\mathcal C$ is of type (II) for all point $(t,x) \in \mathcal C,\ t\in (\theta, t_0].$
\end{thm}
\begin{proof} The proof of this theorem is similar to the one of Thm. 2.11 in \cite{nh3}.
\end{proof}

For a locally Lipschitz function, the usage of notions of sub- and superdifferential as well as reachable gradients to study its differentiability is effective. We use Theorems \ref {cha1}, \ref{di} to establish a relationship between $\ell(t_0,x_0)$ and the set of reachable gradients. 

\begin{defn} Let $v=v(t,x): \ \Omega \to \R$ and let $(t_0,x_0)\in \Omega.$ 
\smallskip

We define the set $D^*v(t_0,x_0)$ of {\it reachable gradients} of $v(t,x)$ at $(t_0,x_0)$ as follows:

$\R^{n+1} \ni (p,q)\in D^*v(t_0,x_0) $ if and only if there exists a sequence $(t_k,x_k)_k\subset \Omega\setminus \{(t_0,x_0)\}$ such that $v(t,x)$ is differentiable at $(t_k,x_k)$ and,
$$(t_k,x_k)\to (t_0,x_0), \ (v_t(t_k,x_k), v_x(t_k,x_k))\to (p,q)\  \text{ as}\  k\to \infty.$$  
\end{defn}

If $v(t,x)$ is a locally Lipschitz function,  then $D^*v(t,x) \ne \emptyset$  and it is a compact set, see \cite{cs}, p.54.
\smallskip

Now let $u(t,x)$ be the Hopf-Lax formula defined by (\ref {hl1}) and let $(t_0,x_0)\in \Omega.$ We denote by 
$$\mathcal H(t_0,x_0)=\{(-H(q), q)\}\ \text{where}\  q=H^*_z(\frac{x_0-y_0}{t_0})\in D^-\sigma (y_0),\ \text{and}\  y_0\in\ell(t_0,x_0).$$
 Then a relationship between $D^*u(t_0,x_0)$ and the set $\ell (t_0,x_0)$ is given by the following theorem.

\begin{thm}\label{rg}
Assume (H0) and (H1). Let $u(t,x)$ be the Hopf-Lax formula for problem (2.1)-(2.2). Then for all $(t_0,x_0)\in \Omega,$ we have
$$D^*u(t_0,x_0)=\mathcal H(t_0,x_0).$$
\end{thm}

\begin{proof}

Let $(p_0,q_0)$ be an element of $\mathcal H(t_0,x_0),$ then $p_0=-H(q_0)$  where $q_0=H^*_z(\frac{x_0-y_0}{t_0})\in D^-\sigma(y_0)$ for some $ y_0\in \ell(t_0,x_0).$ Let $\mathcal C$ be the characteristic curve of type (I) at $(t_0,x_0)$ starting form $y_0.$ By Theorem (\ref{cha1}), the solution $u(t,x)$ is differentiable at all points $(t,x)\in \mathcal C, \ t\in [0, t_0).$ Put $t_k=t_0-1/k, $ then $\mathcal C\ni (t_k,x_k) \to (t_0,x_0)$ and $(u_t(t_k,x_k),u_x(t_k,x_k))=(-H(q_0),q_0)\to (-H(q_0), q_0)\in D^*u(t_0,x_0)$ as $k\to \infty.$  Therefore, $\mathcal H(t_0,x_0)\subset D^*u(t_0,x_0).$ 

On the other hand, let $(p,q)\in D^*u(t_0,x_0)$ and $(t_k,x_k)_k\subset \Omega\setminus \{(t_0,x_0)\}$ such that $u(t,x)$ is differentiable at $(t_k,x_k)$ and
$(t_k,x_k)\to (t_0,x_0), \ (u_t(t_k,x_k), u_x(t_k,x_k))\to (p,q)\  \text{as}\  k\to \infty.$ For each $k\in \N,$ we take $y_k\in \ell(t_k,x_k) $ such that $ (u_t(t_k,x_k), u_x(t_k,x_k))= (-H(q_k),q_k)$ where $q_k=H^*_z(\frac{x_k-y_k}{t_k})\in D^-\sigma(y_k).$ Since $(y_k)_k$ is bounded, we can assume that $y_k$ converges to $y_0.$ Since multivalued function $\ell(t,x)$ is u.s.c, then letting $k\to \infty,$ we see that $y_0\in \ell(t_0,x_0)$ and $p=\lim_{k\to \infty} -H(q_k) =-H(q).$ Thus $(p,q)\in \mathcal H(t_0,x_0).$ The theorem is then proved.
\end{proof}

\begin{rem}
From Theorems \ref{di} and \ref{rg}, we see that if $u(t,x)$ is differentiable at $(t,x)$ then
$$Du(t,x)=(u_t(t,x), D_xu(t,x))=(H^*(\frac {x-y_0}{t})-\frac 1{t}\langle x-y_0,H^*_z(\frac{x-y_0}{t})\rangle,H^*_z(\frac{x-y_0}{t}),$$ for the unique $y_0\in \ell(t,x).$

If $u(t,x)$ is not differentiable at $(t,x),$ then $D^+u(t,x)={\rm co}\mathcal H(t,x)$ and $D^-u(t,x)=\emptyset.$ Therefore it may be helpful to use this result to study some properties of derivatives of the solution of the problem (\ref{2.1})-(\ref{2.2}) defined by Hopf-Lax formula.

\end{rem}

\section{Regularity of Hopf-Lax formula}

In this section we will study the strips of the form $V=(0,t_*)\times \R^n\subset \Omega$ such that $u(t,x)$ is continuously differentiable on them. 

\begin{thm} \label{di1} Assume (H0), (H1). Let $u(t,x)$ be the viscosity solution of problem (\ref{2.1})-(\ref{2.2}) defined by Hopf-type formula (\ref{hl1}). Suppose that there exists $t_* \in (0,T)$ such that the mapping: $y\mapsto x(t_*,y)=y+t_* H_p(\sigma_y (y))$ is injective. Then $u(t,x)$ is continuously differentiable on the open strip $(0,t_*)\times \R^n.$ 
\end{thm}
\begin{proof}
 
Let $(t_0,x_0)\in (0,t_*)\times \R^n$ and let $\mathcal C:$
$$ x=x_0+(t-t_0) H_p(p_0)$$
where $p_0\in D^\#\sigma(y_0),\ y_0\in \ell(t_0,x_0),$  be the characteristic curve going through $(t_0,x_0)$ defined as in Corollary 3.2.

Let $(t_*,x_*)$ be the intersection point of $\mathcal C$ and plane $\Delta^{t_*} :\ t=t_*.$ By assumption, the mapping $y\mapsto x(t_*,y)$ is injective and $\ell(t_*,x_*)\ne \emptyset,$ so there is unique  a characteristic curve passing $(t_*,x_*).$ This characteristic curve is exactly $\mathcal C.$ Therefore, we can rewrite $C$ as follows:
$$ x=x_*+(t_*-t )H_p(p_*)$$
where $p_*\in D^-\sigma(y_*)$ is unique defined with respect to $y_*\in \ell(t_*,x_*).$

Since $\ell^*(t_*,x_*)$ is a singleton, so is $\ell(t_*,x_*).$ Consequently, $\mathcal C$ is of type (I) at $(t_*,x_*)$ and $\ell(t,x) =\{y_*\}$ for all $(t,x)\in (0,t_*)\times \R^n,$  particularly at $(t_0,x_0)$ and then, $y_*=y_0.$ Applying Theorems \ref{di}, 2.8 b) we see that $u(t,x)$ is of class $C^1$ in $(0,t_*)\times \R^n.$
\end{proof}

Note that at some point $(t_0,x_0) \in \Omega$ where $u(t,x)$ is differentiable there may be more than one characteristic curve goes through, that is $\ell^*(t_0,x_0)$ may not be a singleton. Next, we have:
\medskip

\begin{thm} Assume (H0), (H1). Moreover, let $\sigma$  be Lipschitz and of class $C^1(\R^n).$ Suppose that $\ell(t_*,x)$ is a singleton for every point of the plane $\Delta^{t_*} = \{(t_*,x)\in \R^{n+1}:\  x\in \R^n\},\ 0<t_*\le T.$ Then the function $u(t,x)$ defined
by Hopf-Lax formula (\ref{hl1}) is continuously differentiable on the open strip
 $(0,t_*) \times \R^n.$
\end{thm}

\begin{proof}  Let $(t_0,x_0) \in (0,t_*)\times \R^n.$
Since $\sigma(x)$ is Lipschitz  on $\R^n$ then there exists $m>0$ such that $|\sigma_y(y)| \le m$ for all $y\in \R^n.$
\smallskip

By assumption, for each $z\in \R^n,$ there exists unique $p(z)\in \ell(t_*,z)$ such that $\mathcal C_z :\ x(t) =z+(t-t_*)H_p(\sigma (p(z))$ is the unique characteristic curve of type (I)  going through $(t_*, z).$ Since the multivalued function  $z\mapsto \ell(t_*,z)$ is u.s.c; see \cite{vht} and, by assumption $\ell(t_*,z)=\{p(z)\} $ is a singleton for all $z\in \R^n,$ then $z\mapsto p(z)$ is continuous.

Consider the following mapping $\Lambda:\ \R^n \to \R^n$ defined by
$$\Lambda (z)=x_0-(t_*-t_0) H_p( \sigma_y(p(z)), \ \forall \ z\in \R^n,$$
where $p(z)\in \ell(t_*,z).$  Then $\Lambda (z)$ is also continuous on $\R^n.$
\smallskip

Since $\sigma_y(p(z))$ is bounded and $H_p(p)$ is continuous, there exists $M>0$ such that
$$|\Lambda (z)-x_0|\le |t_*-t_0| |H_p(\sigma_y(p(z))| \le M,\ \forall z\in \R^n.$$
Therefore $\Lambda$ is a continuous function from the closed ball $B'(x_0,M)$ into itself. By Brouwer theorem, $\Lambda $ has a fixed point $x_*\in B'(x_0,M), $ i.e., $\Lambda (x_*)=x_*,$ hence 
$$x_0=x_*+(t_*-t_0) H_p(\sigma_y(p(x_*))$$

In other words, there exists a characteristic curve $C$ of the type (I) at $(t_*,x_*)$ described as in Theorem \ref{cha1} passing $(t_0,x_0)$. Since $\ell(t_*,x_*)$ is a singleton, so is $\ell(t_0,x_0).$ Applying Theorem \ref{di}, we see that $u(t,x)$ is continuously differentiable in $(0,t_*)\times\R^n.$
\end{proof}

We note that the hypotheses of above theorems are equivalent to the fact that, there is a unique characteristic curve of type (I) at points $(t_*,x_*),\ x_*\in \R^n$ going through $(t_0,x_0).$ This is also equivalent to the fact that the function $\varphi(x) =u(t_*,x)$ is differentiable on $\R^n.$ In general,
at some point $(t_0,x_0) \in (0,t_*)\times \R^n$ where $u(t,x)$ is
differentiable there may be more than one characteristic curves of
type (I) or (II) at point $(t_*,x_*)$ as above going through $(t_0,x_0),$  that is
$\ell^*(t_*,x_*)$ may not be a singleton. Even neither is
$\ell(t_*,x_*).$ Nevertheless, we have:
\medskip

\begin{thm} Assume (H0), (H1).  Let $u(t,x)$ be the viscosity solution of the problem (\ref{2.1})-(\ref{2.2}) defined by Hopf-Lax formula. Suppose that there exists $t_* \in (0,T)$ such that all characteristic curves passing $(t_*,x),\ x\in \R^n$ are of type (I). Then $u(t,x)$ is continuously differentiable on the open strip $(0,t_*)\times \R^n.$
\end{thm}

\begin{proof}
We argue similarly to the proof of Theorem \ref{di1}. Let $(t_0,x_0)\in (0,t_*)\times \R^n$ and let $\mathcal C:$
$$ x=x_0+(t-t_0) H_p(p_0)$$
where $p_0\in D^-\sigma(y_0),\ y_0\in \ell(t_0,x_0)$  be the characteristic curve going through $(t_0,x_0)$ defined as above.
\smallskip

Let $(t_*,x_*)$ be the intersection point of $\mathcal C$ and plane $\Delta^{t_*} :\ t=t_*.$  Then we have
$$ x_*=x_0+(t_*-t_0) H_p(p_0)$$
Therefore, we can rewrite $\mathcal C$ as
$$ x=x_*-(t_*-t_0)H_p(p_0) +(t-t_0) H_p(p_0) = x_*+(t-t_*) H_p(p_0)$$
to see that it is also a characteristic curve passing$(t_*,x_*).$ By assumption, $\mathcal C$ is of type (I) at this point, so it is of type (I) for all $(t,x)\in \mathcal C, \ 0\le t <t_*,$ by Theorem \ref{cha1}. Thus, $\ell(t_0,x_0)$ is a singleton. As before, we come to the conclusion of the theorem.
\end{proof}

The following corollary establishes the existence of a strip of differentiability of viscosity solution of the problem (\ref{2.1})-(\ref{2.2}) defined by Hopf-Lax formula.

\begin{cor}
Assume (H0), (H1). In addition, suppose that the function $H=H(p)$ is both uniformly convex and semiconcave; $\sigma$ is semiconvex. Then there is the greatest number $t_*\in (0, T]$ such that the Hopf-Lax formula is of class $C^1((0,t_*)\times\R^n).$ 
\end{cor}

\begin{proof} Since $H$ is uniformly convex, then for each $t\in (0,T],$ the function $u(t,\cdot)$ where $u(t,x)$ defined by Hopf-Lax formula, is a semiconcave, see \cite{ev}. On the other hand, by Theorem \ref{scc}, there exists $t_0\in (0,T]$ such that $u(t_0,\cdot)$ is semiconvex. Therefore, $u(t_0,\cdot)$ is differentiable on the plane $t=t_0,$ and in virtue of Theorem 4.2, $u(t,x)$ is continuously differentiable on the strip $(0,t_0)\times \R^n.$
\smallskip

Let $t_*=\sup\{t\in (0,T]\; | \; u(t,\cdot)\  \text{is differentiable on}\ \R^n\}.$ Then arguing as above, we see that $u(t,x)\in C^1((0,t_*)\times\R^n).$
\end{proof}

\begin{rem}

 1. If $H\in C^2(\R^n)$ and there is $\alpha >0$ such that $\frac 1\alpha I\le D^2H\le \alpha I$ where $D^2H$ is the Hessian of $H$ and $I$ is the $(n\times n)$-unit matrix, then $H$  is both uniformly convex and semiconcave.

\smallskip
2. The above result can be considered as a version of the existence of the classical  solution of the problem (\ref{2.1})-(\ref{2.2}) under assumption that the data are only of class $C^1.$ (cf. Corollary 1.5.5, \cite{cs}.)
\end{rem}

{\bf An Application:} {\it  When backward solution is forward solution}
\smallskip

It should be noted that, a function $u(t,x)\in C^1([0,T]\times\R^n)$ is a classical solution of  the equation $u_t+H(t,x,D_xu)=0$ if and only if, the function $v(t,x)=u(T-t,x)$ is a classical solution of the equation $v_t-H(T-t,x,D_xv)=0.$ This statement is not true in general, if $u(t,x)$ is merely a viscosity solution of the equation.

\smallskip
In \cite{bcjs} the authors introduce the notion ``forward'' and ``backward'' viscosity solution of a Hamilton-Jacobi equation of the form $u_t+H(t,x,Du)=0$ as follows. The function $u(t,x)\in C([0,T]\times \R^n)$ is called a forward solution of the equation $u_t+H(t,x,D_xu)=0$ if it is a viscosity in the usual sense, while $u(t,x)$ is said to be a backward solution of the equation, if $v(t,x)=u(T-t,x)$ is a viscosity solution of the equation $v_t-H(T-t,x,D_xv)=0.$

\smallskip
For a comprehensive knowledge of this topic, we refer the readers to the paper just mentioned above. Here we would like review the case $u_t +H(Du)=0$ where the viscosity solution is defined by Hopf-Lax formula.

\smallskip
Following the presentation in \cite{bcjs}, sect. 4, let function $w(t,x)$ be a backward viscosity solution of the problem $w_t +H(Dw)=0$ with terminal condition $w(T,x)=g(x),$ i.e., the function $v(t,x)=w(T-t,x)$ is the viscosity solution of the problem $v_t -H(Dv)=0$ with initial condition $v(0,x)=g(x).$ We reproduce the following theorem and give another proof for continuous differentiability of $u(t,x)$ in the statement (ii). Note that, the origin proof is rather long but direct, ours seems to be ``shorter'' since it is inherited several intermediate results.

\begin{thm}\label{fb}
Assume $H:\ \R^n\to \R$ is of class $C^1(\R^n),$ convex and superlinear, $g\in C^1(\R^n)$ and is Lipschitz. Let $w(t,x)$ be the backward viscosity solution  of the problem $w_t+H(Dw)=0; \ w(T,x)=g(x)$ and let $u(t,x)$ be the (forward) viscosity solution of the problem $u_t+H(Du)=0; \ u(0,x)=w(0,x).$ Then
\smallskip

(i) A necessary and sufficient condition for $u(T,x)=g(x)$ is
\begin{equation}\label{bf} g(x)=\min_z\max_y\Big \{g(y)-TH^*\Big (\frac{y-z}{T}\Big)+TH^*\Big(\frac{x-z}{T}\Big)\Big\}.\end{equation}

\smallskip
(ii) If (\ref{bf}) holds, then $u(t,x)=w(t,x),\ (t,x)\in [0,T]\times \R^n$ and $u$ is continuously differentiable on $(0,T]\times \R^n.$
\end{thm}

\begin{proof} The proof of (i) and equality $u(t,x)=w(t,x)$ is taken from \cite{bcjs}. We only check that the function $u=u(t,x)$ is continuously differentiable in $(0, T]\times\R^n.$
Let $u(t,x)$ be a unique viscosity of the problem (2.1)-(2.2), where $\sigma(x)=w(0,x)$ defined by Hopf-Lax formula. By assumption, $u(T,x)=g(x)$ is a differentiable function on $\R^n,$ we apply Theorem 4.2 to get the desire result.
\end{proof}

\begin{cor}
Let the assumptions of Theorem \ref{fb} hold and let $g(x)$ be a semiconvex function. Suppose that $H\in C^2(\R^n)$ and is  uniformly convex. If $u(T,x)=g(x)$ then $u(t,x)=w(t,x),\ (t,x)\in [0,T]\times \R^n$ and $u$ is continuously differentiable on $(0,T]\times \R^n.$
\end{cor}

\begin{proof} Following \cite{ev}, $u(t,x)$ is semiconcave in $x$ for fixed $t\in (0,T].$ Thus, $u(T,x)=g(x)$ is both semiconvex and semiconcave, then $u(T,x)$ is of class $C^1(\R^n).$ Applying Theorem 4.2, we see that $u(t,x)\in C^1((0,T])\times\R^n).$
\end{proof}

\end{document}